\documentclass[11pt]{amsart}
\usepackage{amsmath,amssymb,amsfonts,amsthm,amscd,indentfirst,graphicx}
\usepackage[dvipdfm]{hyperref}

\numberwithin{equation}{section}

\newtheorem{prop}{Proposition}[section]
\newtheorem{theo}[prop]{Theorem}
\newtheorem{lemm}[prop]{Lemma}
\newtheorem{coro}[prop]{Corollary}
\newtheorem{rema}[prop]{Remark}

\newtheorem{defi}[prop]{Definition}

\def\begeq{\begin{equation}}
\def\endeq{\end{equation}}

\def\and{\quad{\rm and}\quad}

\def\<{\langle}
\def\>{\rangle}

\def\e{\epsilon}

\def\di{\displaystyle}
\def\Dint{\displaystyle\int}

\begin{document}
\title{A geometric characterization of a sharp Hardy inequality}
\author{Roger T. Lewis, Junfang Li, and Yanyan Li}
 \address{Department of Mathematics\\
         University of Alabama at Birmingham\\
         Birmingham, AL 35294}
 \email{rtlewis@uab.edu}

 \address{Department of Mathematics\\
         University of Alabama at Birmingham\\
         Birmingham, AL 35294}
 \email{jfli@uab.edu}

\address{Department of Mathematics\\
         Rutgers University\\
         New Brunswick, NJ 08903}
\email{yyli@math.rutgers.edu}

\thanks{Research of the second author was supported in part by NSF DMS-1007223. Research of the third author was supported in part by NSF DMS-0701545.}

\begin{abstract}
  In this paper, we prove that the distance function of an open connected set in $\mathbb R^{n+1}$ with a $C^{2}$ boundary is superharmonic in the distribution sense if and only if the boundary is {\em weakly mean convex}. We then prove that Hardy inequalities with a sharp constant hold on {weakly mean convex} $C^{2}$ domains. Moreover, we show that the {weakly mean convexity} condition cannot be weakened. We also prove various improved Hardy inequalities on mean convex domains along the line of Brezis-Marcus \cite{BM}.
 \end{abstract}


\maketitle

\section{Introduction}\label{introduction}

When $n\ge2$, the well-known Hardy inequality states that
\begin{equation}
  \left (\frac{n-1}2\right )^2\di\int_{\mathbb R^{n+1}}\frac{|u( x )|^2}{| x |^2}d x
  \le \di\int_{\mathbb R^{n+1}}|\nabla u|^2d x ,\quad u\in
  C^\infty_0(\mathbb R^{n+1}),
  \label{hardy}
\end{equation}
where $C^\infty_0(\mathbb R^{n+1})$ denotes the set of 
$C^\infty$ functions on $\mathbb R^{n+1}$ with compact support.

For
 domains with boundaries, Hardy's inequality can be formulated in terms of the distance function from points in the domain to the boundary.
In this paper, a domain is an open connected subset of
a Euclidean space.
 The following Hardy-type inequality on domains has been studied by several authors:
\begin{equation}
  \di\int_{\Omega}|\nabla f( x )|^pd x \ge c(n,p,\Omega)\di\int_{\Omega}\frac{|f(x)|^p}{\delta( x )^p}d x , \quad f\in C^{\infty}_0(\Omega),
  \label{hardy domain}
\end{equation}
where $\Omega\subset \mathbb R^{n+1}$ is a domain with nonempty
boundary, $n\ge 1$, $1<p<\infty$, and $\delta( x ):=\inf_{ y \in \mathbb{R}^{n+1}\setminus \Omega} dist( x , y )$. For a {\em convex domain} $\Omega$ in $\mathbb R^{n+1}$, $n\ge 1$, the best constant is
\begin{equation}
  c(n,p,\Omega)=\left (\frac{p-1}p\right)^p,
\end{equation}
see \cite{MMP} and \cite{MS}.

When $n\ge2$, many results for the Hardy inequality assume that the domain is convex.
However, there are indications that the Hardy inequality should hold for non-convex domains as well.
Filippas, Maz'ya, and Tertikas \cite{FMT1} proved that in a small enough tubular neighborhood of the boundary of a
bounded domain, a Hardy-Sobolev inequality holds. In \cite{FMT2}, they showed that a Hardy-Sobolev inequality holds if a bounded $C^2$ domain $\Omega\subset \mathbb R^{n+1}$, $n\ge2$, satisfies the condition
\begin{equation}\label{global condition 2}
  -\Delta \delta({ x })\ge0,
\end{equation}
(see Theorem 1.1 (i) and condition (C) in \cite{FMT2}) .

Filippas, Moschini, and Tertikas \cite{FMoT} proved an improved Hardy inequality for domains satisfying
\begin{equation}
  -\mbox{div}(| x |^{1-n}\nabla \delta( x ))\ge 0,\quad \mbox{a.e. in}\ \Omega,
  \label{global condition}
\end{equation}
see another proof in \cite{AL}.

Notice that conditions (\ref{global condition 2}) and (\ref{global
condition}) are global conditions. Namely, they depend on the
property of the whole domain which can make them hard to verify.
As a consequence, there are few known non-convex examples for
application. In fact, the only examples stated in \cite{FMoT}
satisfying condition (\ref{global condition}) are balls $B_R$.
Convex domains are known to satisfy condition (\ref{global
condition 2}). For non-convex domains, a ring torus is shown to
satisfy (\ref{global condition 2}) by Armitage and Kuran in
\cite{AK}. The superharmonicity of the same example has been shown
to hold off a measure zero set in a recent work of Balinsky,
Evans, and Lewis \cite{BEL}. For other non-convex domains, Hardy
type inequalities are proved to be true for small enough tubular
neighborhood of a surface \cite{FMT1} and convex domains with
punctured balls \cite{AL}.

Clearly the convexity assumption is very restrictive.
On the other hand, there are smooth bounded domains on which Hardy's inequality fails with the sharp constant,
(see \cite{MS,MMP}). It has been an outstanding question as to whether there is a more
general criteria for domains for
which a sharp Hardy inequality holds.

We will give an affirmative answer to this question. To illustrate the main idea, we first recall for a domain $\Omega\subset \mathbb R^{n+1}$ with $C^2$ boundary $\partial\Omega$, the principal curvatures with respect to the outward unit normal
\[\kappa=(\kappa_1,\cdots,\kappa_n)\]
at a point on the boundary are defined as the eigenvalues of the second fundamental form with respect to the induced metric. It is well-known that a bounded domain with a $C^2$ boundary is strictly convex if and only if $\kappa_i>0$ for each $i=1,\cdots,n$, (see e.g. chapter 13 of \cite{Th}).

The trace of the second fundamental form is defined as the mean curvature $H=\sum_i\kappa_i$, where we adopt the convention that a standard unit sphere $\mathbb S^n\subset\mathbb R^{n+1}$ has mean curvature $n$ everywhere. We now recall the definition of a mean convex domain.

\begin{defi}
  Suppose $\Omega\subset \mathbb R^{n+1}$ is 
a domain 
 with $C^2$ boundary $\partial \Omega$. We say $\Omega$ or
  $\partial\Omega$ is (strictly) {\em mean convex}, if the mean curvature $H( y )>0$ for all
  $ y \in \partial \Omega$; and {\em weakly mean convex}, if $H( y )\ge0$, for all
  $ y \in \partial \Omega$.
\end{defi}

The mean convexity condition is a much weaker condition than the convexity condition since the fundamental group of a convex domain has to be trivial while for a mean convex domain it may be non-trivial. For example, a ring torus with minor radius $r$ and major radius $R$ satisfying $R>2r$ has
positive mean curvature $H>0$ everywhere. When $R=2r$, this ring torus is called a {\em critical} ring torus.
Other non-convex examples include a small perturbation of the above ring torus, a torus with
high genus, a long cow horn, etc. Another highly interesting surfaces from differential geometry are minimal surfaces
which have $H\equiv 0$ everywhere and may possess rich topological and geometric structure.\\

There has been an increasing amount of attention in recent studies of partial differential equations
and associated inequalities on $\Omega\subset\mathbb{R}^n$ devoted to the effects of curvature of the boundary
$\partial\Omega$. In particular, the important role of the mean curvature for points on $\partial\Omega$ has
been investigated recently, e.g., see Harrell \cite{H}, Harrell and Loss \cite{HL}, and Ghoussoub and Robert \cite{GR}.
Curvature-induced
bound states in quantum wave guides arise in work of Duclos and Exner \cite{DE}.  More recently in \cite{BEL},
curvature is shown to be an important consideration in the study of Hardy-type inequalities. We continue those studies here.

{\it 
From now on, unless otherwise stated,
we use $\Omega\subset \mathbb R^{n+1}$ to denote
a domain with   $C^{2}$ boundary, $n\ge 1$.
}

We now state one of our main theorems in this paper.

\begin{theo}\label{sharp hardy inequality}
  Suppose $\Omega\subset\mathbb R^{n+1}$ is weakly mean convex, then for any $f\in C^{\infty}_0(\Omega)$, with $p>1$, the following holds
    \begin{equation}
  \di\int_{\Omega}|\nabla f( x )|^pd x \ge c(n,p,\Omega)\di\int_{\Omega}\frac{|f(x)|^p}{\delta( x )^p}d x ,
    \label{sharp hardy inequality : equ1}
  \end{equation}
  where $c(n,p,\Omega)=(\frac{p-1}p)^p$ is the best constant. Moreover, equality in (\ref{sharp hardy inequality : equ1}) can not be achieved by non-zero functions.\\
\end{theo}

In general, the best constant in (\ref{sharp hardy inequality : equ1}) for $p>1$ is given by
\begin{equation}
  \mu_p(\Omega):=\di\inf_{u\in W_0^{1,p}(\Omega)}\frac{\int_{\Omega}|\nabla u|^p}{\int_{\Omega}|u/\delta|^p}
  \label{best constant : hardy}
\end{equation}
which we denote by $\mu(\Omega)$ when $p=2$. For convex domains
$\mu(\Omega)=1/4$, (see \cite{MS,MMP}), but there are smooth bounded domains such that $\mu(\Omega)<1/4$.
For smooth bounded domains and $p=2$, the infimum in (\ref{best constant : hardy}) is achieved if and only if $\mu(\Omega)<1/4$.

For a bounded domain with $C^2$ boundary $\partial \Omega$ we know that $\mu_p(\Omega)\le (\frac{p-1}p)^p$ (see \cite{MMP}). On the other hand, inequality (\ref{sharp hardy inequality : equ1})
implies that $\mu(\Omega)\ge (\frac{p-1}p)^p$ for weakly convex domains. As a consequence of Theorem \ref{sharp hardy inequality},
we now know that $\mu(\Omega)=(\frac{p-1}p)^p$ for bounded {\em weakly mean convex} domains with a $C^{2}$ boundary.

This boundary geometric condition is also sharp in the sense that the condition $H\ge0$ can not be weakened. Explicit examples are constructed in section \ref{main results} showing that the sharp Hardy inequality fails if the boundary condition is weakened to $H\ge -\epsilon$, for any $\epsilon>0$.

Moreover, neither the diameter nor the interior radius of the domain $\Omega$ in (\ref{sharp hardy inequality : equ1}) need to be bounded. Many of the previous theorems need to assume that the domain is either bounded or the interior radius is bounded. \\

In this paper, we will also prove a Brezis-Marcus type of improved Hardy inequality. In the case $p=2$, Brezis and Marcus \cite{BM} proved the following inequality for bounded domains with $C^2$ boundary
\begin{equation}
  \di\int_{\Omega}|\nabla u|^2d x \ge \frac{1}{4}\di\int_{\Omega}(u/\delta)^2d x +\Lambda\int_{\Omega} u^2d x , \forall u\in H^1_0(\Omega),
  \label{BM}
\end{equation}
where $\Lambda$ is the best constant defined as
\begin{equation}
  \Lambda : = \di\inf_{\int_{\Omega}f^2d x =1}\Big[\int_{\Omega}|\nabla f|^2d x -\frac{1}{4}\int_{\Omega}\frac{f^2}{\delta^2}d x \Big].
  \label{BM best}
\end{equation}

When $\Omega$ is a {\em convex domain}, they showed that $\Lambda\ge\lambda_{BM}:= \frac{1}{\mbox{4 diam}^2(\Omega)}$ which gave an improved Hardy inequality with a positive remainder term. Along this line, there have been intensive studies on improved Hardy type inequalities recently, see e.g. \cite{BM,BFT,DD,HHL,FT,VZ,T,FMT,FMoT,AL,EL,BEL} and the references therein. For the most part the estimates are given for convex domains. For example, Hoffmann-Ostenhof, M., Hoffmann-Ostenhof, Th., and Laptev \cite{HHL} proved $\Lambda\ge\lambda_{HHL}:= \frac{c(n)}{|\Omega|^\frac{2}{n+1}}$ for $n\ge1$, where $c(n)=\frac{(n+1)^\frac{n-1}{n+1}|\mathbb S_{n}|^\frac{2}{n+1}}{4}$ and $|\mathbb S_n|$ is the area of the unit sphere; Filippas, Maz'ya, and Tertikas \cite{FMT} proved $\Lambda\ge\lambda_{FMT}:= \frac{3}4 R_{int}^{-2}$ for $n\ge2$, where $R_{int}:=\sup_{ x \in \Omega}\delta( x )$; Evans and Lewis \cite{EL} proved $\Lambda\ge \lambda_{EL}:=6\lambda_{HHL}$; Avkhadiev and Wirths \cite{AW} proved $\Lambda\ge\lambda_{AW}:=j_0^2 R_{int}^{-2}$ where $j_0=0.940\cdots$ is the first positive root of an equation of Bessel's function. Results in \cite{HHL}, \cite{FMT}, \cite{EL} and \cite{AW} improved the estimate for $\Lambda$ in \cite{BM}.

Below we will give an improved inequality on {\em weakly mean convex} domains along the line of Brezis-Marcus.

\begin{theo}\label{main theorem}
{\em( Improved Hardy-Brezis-Marcus Inequality)} 
  Suppose $\Omega\subset\mathbb R^{n+1}$ is {\em weakly mean convex} and assume that
$H_0:=\inf_{ x \in\partial\Omega}H(x)\ge0$,
then for any $f\in C^{\infty}_0(\Omega)$
    \begin{equation}
      \Dint_{\Omega}|\nabla f|^2 d x \ge\di \frac{1}{4}\Dint_\Omega\frac{|f|^2}{\delta^2}d x +\lambda(n,\Omega)\Dint_\Omega |f|^2d x,
    \label{main theorem : equ2}
  \end{equation}
  where
  $\lambda(n,\Omega)=\di\inf_{ x \in\Omega}\frac{-\Delta\delta( x )}{2\delta( x )}\ge\frac{2}{n}H^2_0$.
\end{theo}

The $L^p$ version of this theorem is stated in Theorem \ref{improved hardy : Lp}. The constant $\lambda(n, \Omega)$ in Theorem \ref{main theorem} depends on $\Omega$. In general, $\lambda(\Omega)>\frac{2}{n}H^2_0$, but we will show that if $\Omega$ is a ball, then $\lambda(\Omega)=\frac{2}{n}H^2_0$. More specifically, we have the following corollary of Theorem \ref{main theorem}.

\begin{coro}
  \label{on balls}
For any $f\in C^{\infty}_0(B_R)$, the following holds:
  \begin{equation}
    \begin{array}{rll}
      \Dint_{B_R}|\nabla f|^2 d x \ge&\di \frac{1}{4}\Dint_{B_R}\frac{|f|^2}{\delta^2}d x +\lambda(n,R)\Dint_{B_R}|f|^2d x,
    \label{on balls : equ2}
  \end{array}
  \end{equation}
  where $\lambda(n,R)=\frac{2n}{R^2}$.
\end{coro}

In the general weakly mean convex case, it is possible that $H_0$ is zero on some subsets of the boundary, but $\lambda(n, \Omega)$ is still strictly positive. Consider the critical ring torus example with major radius $R=2$ and minor radius $r=1$. Direct calculations show that mean curvature on the inner equator is $H\equiv0$ but $\lambda(n, \Omega)=1$. More details can be found in Example 1 section \ref{examples}.

Other extreme examples, which may be of independent interest, are domains with embedded minimal surfaces as boundary.

\begin{coro}
  \label{minimal surface : coro}
Let $\Omega\subset\mathbb R^{3}$ be an open connected set which has an embedded minimal surface $\mathcal M$ as the boundary, i.e., $H(\mathbf y)\equiv 0$ for any $\mathbf y\in \mathcal M$. Let $\kappa_0:=\di\inf_{\mathbf y\in \mathcal M}|\kappa(\mathbf y)|$ be the infimum of the absolute value of all the principal curvatures. Then for any $f\in C^{\infty}_0(\Omega)$, we have the following.
    \begin{equation}
      \Dint_{\Omega}|\nabla f|^2 d\mathbf x\ge\di
      \frac{1}{4}\Dint_\Omega\frac{|f|^2}{\delta^2}d\mathbf x+\lambda(n,\Omega)\Dint_\Omega |f|^2d\mathbf x,
    \label{}
\end{equation}
where $\lambda(n, \Omega)=\kappa_0^2$.
\end{coro}

The proof of Corollary \ref{minimal surface : coro} can be found in section \ref{examples}.\\

One of the key 
observations of this paper is the following theorem. 
We believe it is also of independent interest.
(See section~2 for the definition of the near point and good set $G$.)
\begin{theo}
  \label{key theorem}
Let $n\ge1$, $\Omega\subset \mathbb R^{n+1}$
  and $\delta( x ):=\inf_{ y \in \mathbb{R}^{n+1}\setminus \Omega} dist(x,y)$. Then
\begin{equation}
  -\Delta \delta( x )\ge \frac{nH( y )}{n-\delta H( y )},
  \label{key theorem : equation}
\end{equation}
in the distribution sense: for any $\varphi\in C^{\infty}_0(\Omega)$, $\varphi\ge0$, we have
  \begin{equation}
    \Dint_{\Omega}\nabla\delta\nabla \varphi dx\ge \Dint_{\Omega}\frac{nH}{n-\delta H}\varphi dx,
    \label{distribution sense}
  \end{equation}
where $H( y )$ is the mean curvature at the nearest
 point $y=N( x )\in \partial \Omega$ for points $ x \in G$.
\end{theo}

It is well-known that the Hessian of the distance function is positive definite for a convex domain, see, e.g., \cite{GT}.
However, to prove Hardy-type inequalities, the full strength of a positive Hessian is not needed. Only the Laplacian of $\delta( x )$ is involved. Using Theorem \ref{key theorem}, we can reduce the global
superharmonicity condition of the distance function to a geometric boundary condition which has been intensively studied in differential geometry.

Armitage and Kuran \cite{AK} proved that
$\delta( x )$ is superharmonic if the domain is convex. They also showed by examples that the converse is not true
when $n>1$.

Moreover, we have the following equivalence theorem which states that the superharmonicity of the
Laplacian of the distance function can be uniquely characterized by the boundary mean curvature.

 \begin{theo}
   \label{equivalence theorem}{\bf(Equivalence Theorem)}
 Let $\Omega\subset\mathbb R^{n+1}$
  and $\delta( x )$ be the
 distance function to the boundary. Then $\delta( x )$ is a superharmonic function on $\Omega$ off the singular set $S$ if and only if
 $\partial \Omega$ is {\em weakly mean convex}, where $S$ is defined in (\ref{singular set}).\\
 \end{theo}

\begin{rema}
  When $n=1$, it is well-known that mean convexity is equivalent to convexity. A more general equivalence result is stated in Proposition \ref{prop : inf equi}.
\end{rema}

Theorem \ref{key theorem} was motivated by recent work of Balinsky, Evans, and Lewis \cite{BEL} as well as Lemma 14.17 of
Gilbarg-Trudinger \cite{GT}. Lemma 14.17 of \cite{GT} was used by Flippas, Maz'ya, and Tertikas \cite{FMT2} to
estimate the upper bound of $|\delta\Delta\delta|$ when the point is close to the boundary, see Condition (R)
in \cite{FMT2}. In \cite{BEL}, a generalization of it was used by Balinsky, Evans, and Lewis to relate the Laplacian of distance function on the whole domain
(except for a set of measure zero) to the boundary principal curvatures.\\

The rest of this paper is organized as follows. In section \ref{distance function and boundary geometry}, we collect
necessary preliminaries and relate the superharmonicity of the distance function to the boundary geometry on points in the domain off the singular set. In section \ref{distribution section}, we prove
 Theorem \ref{key theorem}. In section
\ref{main results}, we give proofs to the main theorems and discuss the sharpness of
the geometric boundary conditions. In section \ref{applications}, we extend other related important inequalities to mean convex domains. In section \ref{examples}, we give non-trivial examples of non-convex domains on which Hardy type inequalities hold. \\

 \section{The distance function and boundary geometry}\label{distance function and boundary geometry}

 Let $\delta( x ):=\inf_{ y \in \mathbb{R}^{n+1}\setminus \Omega} dist(x,y)$ 
denote the distance from a point
 $ x \in\Omega$ to $\partial \Omega$. In this section, we 
recall some properties of this distance function.
For $x\in \Omega$, let
 $N_{\partial \Omega}( x ):=\{ y \in\partial\Omega :| y - x |=\delta( x )\}$ denote the
set of nearest points on $\partial \Omega$.
When $N_{\partial \Omega}( x )$ contains exactly one point,
we denote it as $N(x)$.

This distance function has been extensively studied. The main references we refer to here are \cite{GT, LN}, see also \cite{EE}.  Recall the following definition from Li-Nirenberg \cite{LN}.

\begin{defi}\label{singular set}
Let  $G\subset \Omega$ be the largest 
open subset of $\Omega$ such that for every $x$ in $G$
there is a unique nearest point on $\partial \Omega$ to $x$.
  We  call the complement of the good set $G$
 to be a singular set and denote it as $S=\Omega\setminus G$.
\end{defi}

 We know that $\delta( x )$ is locally Lipschitz continuous, cf. \cite{GT}, hence it is differentiable a.e.. Theorem 5.1.5 (\cite{EE}) implies that $\delta( x )$ is differentiable in $\Omega$ if and only if $N_{\partial\Omega}( x )$ contains only one element. In particular, it is differentiable in $G$. Hence, if $ x \in G$, then $\delta( x )$ is differentiable, $\nabla \delta( x )$ is continuous, and $\nabla \delta( x )= \frac{ x - y }{| x - y |}$ where $ y =N( x )$ is the nearest point.

 We will show that if the boundary $\partial \Omega$ is $C^2$, then $\delta( x )$ is $C^2$ in $G$. Note that a  proof of this result in the much more general setting of Finsler manifolds was given in \cite{LN, LN1}. First we have the following geometric lemma.
\begin{lemm}
  \label{positiveness}
  Let $\Omega\subset \mathbb R^{n+1}$.
 Suppose $ x \in G$ and
  let $ y =N( x )$ be the nearest point of ${x}$ on the boundary. Let
  $\kappa_i( y )$, $i=1,\cdots,n$ be the
  principal curvatures of the boundary at $ y $ with respect to the outward unit normal, then
  \begin{equation}
    1-\delta( x )\kappa_i( y )>0,
    \label{positivity}
  \end{equation}
  for all $x\in G$ and for all $i$.
\end{lemm}
\begin{proof}
 Suppose $ x \in G$. Let $B_{\delta}( x )$ be
 the ball centered at $ x $ with radius $\delta$ satisfying
 $\overline B_{\delta}( x )\cap (\mathbb R^{n+1}\backslash \Omega)=\{ y \}$. We may assume $\kappa_i>0$,
 otherwise the statement is trivial. Recall the principal radius is the reciprocal of principal curvature, i.e.,
 $r_i:=\frac{1}{\kappa_i}$. It is also the radius of the osculating circle. Since the boundary is $C^2$, it is
 geometrically evident that $\delta( x )\le r_i$. Otherwise  $\overline B_{\delta}( x )$ will
 enclose the osculating circle and will intersect the boundary more than once. Equivalently, we know $1-\delta\kappa_i\ge0$.
 On the other hand, if $x\in G$, then $1-\delta(x)\kappa_i>0$.
Indeed, in view of 
Corollary 4.11 of \cite{LN},
there exists  $\epsilon>0$ such that
$$
x_t:= N(x)+[\delta(x)+t]\eta(N(x))\in G,\qquad 0<t\le \epsilon,
$$
for $\eta(N(x)):=-\nu(N(x))$ to be the unit inward normal at $N(x)$ and
$$
\delta(x_t)= \delta(x)+t.
$$
Consequently,
$$
B(x_t, \delta(x)+\epsilon)\subset G,
$$
from which we deduce $
1-\delta(x)\kappa_i>
1-[\delta(x)+\epsilon]\kappa_i\ge 0.
$
\end{proof}

Applying Lemma \ref{positiveness}, one has the following lemmas.

\begin{lemm}
  \label{C2}
   Let $\Omega\subset \mathbb R^{n+1}$.
 Then the distance function $\delta( x )$
   is  in $C^2(G)$.
\end{lemm}

\begin{proof}
The proof of this lemma is by the standard inverse mapping theorem which can be found in Gilbarg-Trudinger \cite{GT}. The original proof was for a small enough tubular
  neighborhood of the boundary and can be found in Lemma 14.16 in \cite{GT}. For reader's convenience, we include the
  proof here and modify it slightly for this setting.

For $y\in\partial \Omega$, we could let $\nu(y)$ and $T(y)$ denote
respectively the unit outward normal to $\partial \Omega$ at $y$
and the tangent hyperplane to $\partial \Omega$ at $y$. By a
rotation of coordinates we can assume that the $x_{n+1}$
coordinate axis lies in the direction $-\nu(y_0)$. In some
neighborhood $\mathcal N$ of $y_0$, $\partial \Omega$ is then
given by $x_{n+1}=\varphi(x')$ where $x'=(x_1,\cdots,x_n)$,
$\varphi\in C^2(T(y_0)\bigcap\mathcal N)$ and $D\varphi(y'_0)=0$.
The eigenvalues of $[D^2\varphi(y'_0)]$,
$\kappa_1,\cdots,\kappa_n$ are call the principal curvatures of
$\partial \Omega$ at $y_0$. By a further rotation of coordinates
the Hessian matrix can be diagonalized to be
  \begin{equation}
    [D^2\varphi(y'_0)]=\mbox{diag}[\kappa_1,\cdots,\kappa_n].
    \label{}
  \end{equation}
  We call the coordinates after the rotation the {\em principal coordinate system} at $y_0$. The unit outward normal vector $\bar{\mathbf \nu}(y')=\mathbf \nu(y)$ at the point $y=(y',\varphi(y'))\in \mathcal N\bigcap \partial \Omega$ is given by
  \begin{equation}
    \nu_i(y)=\frac{D_i\varphi(y')}{\sqrt{1+|D\varphi(y')|^2}}, i=1,\cdots,n, \nu_{n+1}(y)= \frac{-1}{\sqrt{1+|D\varphi(y')|^2}}.
    \label{}
  \end{equation}
 Therefore, under the principal coordinates at $y_0$, we have
 \begin{equation}
   D_j\bar{\mathbf v}_i(y'_0)=\kappa_i\delta_{ij}, i,j=1,\cdots, n.
   \label{}
 \end{equation}

For each point $x\in G $, there exists a unique point $y=y(x)\in \partial \Omega$ such
that $|x-y|=\delta(x)$. The points $x$ and $y$ are related by
  \begin{equation}
    x=y-\delta\nu(y).
    \label{C2 : equ1}
  \end{equation}
We show that this equation determines $y$ and $\delta$ as $C^1$ functions of $x$.

For a fixed point $x_0\in G$, let $y_0=y(x_0)$ and choose a principal coordinate system
at $y_0$. Let $\mathbf g=(g^1,\cdots,g^n)$ be a mapping from $\mathcal U=(T(y_0)\bigcap \mathcal N(y_0))\times \mathbb R$
into $\mathbb R^{n+1}$ by

  \begin{equation}
    \mathbf g(y',\delta)=y-\nu(y)\delta, y=(y',\varphi(y')).
    \label{}
  \end{equation}

Clearly, $\mathbf g\in C^1(\mathcal U)$, and the Jacobian matrix of $\mathbf g$ at $(y'_0, \delta(x))$ is given by
\begin{equation}
  [D\mathbf g]=\mbox{diag}[1-\kappa_1\delta,\cdots, 1-\kappa_n\delta, 1].
  \label{}
\end{equation}

Since the Jacobian of $\mathbf g$ at $(y'_0, \delta(x_0))$ is given by
\begin{equation}
  \det[D\mathbf g]=(1-\kappa_1\delta(x_0))\cdots(1-\kappa_n\delta(x_0))>0,
  \label{}
\end{equation}
because $x\in G$, it follows from the inverse mapping theorem that for some neighborhood
$\mathcal M=\mathcal M(x_0)$ of $x_0$, the mapping $y'$ is contained in $C^1(\mathcal M)$. From (\ref{C2 : equ1}) we have
$D\delta(x)=- \nu(y(x))=-\nu(y'(x))\in C^1(\mathcal M)$ for $x\in \mathcal M$. Hence
$\delta\in C^2(G)$.

\end{proof}

\begin{lemm}
  \label{hessian}
   Let $\Omega\subset \mathbb R^{n+1}$.
 Suppose $ x \in G$ and let
   $ y =N( x )$ be the nearest point on the boundary. Let $\kappa_i( y )$, $i=1,\cdots,n$ be the principal
   curvatures of the boundary at $ y $, then in terms of a principal coordinate system at $ y $, for $\forall x\in G$, we have
   \begin{equation}
     [D^2\delta( x )]={\rm diag}\Big[\frac{-\kappa_1}{1-\delta\kappa_1}, \cdots, \frac{-\kappa_n}{1-\delta\kappa}, 0  \Big],
     \label{Hessian : equ}
   \end{equation}
  where $[D^2\delta( x )]$ is the Hessian matrix of the distance function and right hand side is a diagonal matrix.
\end{lemm}

\begin{proof}
  Geometrically, the result follows from the fact that circles of principal curvature to $\partial \Omega$ at
  $ y _0$ and to the level surface at $ x _0$ are concentric. Since it is already proved that
  $\delta\in C^2(G)$ from Lemma \ref{C2}, using the
  definition of principal curvatures and finding Jacobi matrix under change of variables, the proof of Lemma 14.17 in \cite{GT},
  can be used without any change.
\end{proof}
An expression for the Laplacian of a $C^2(\mathbf{R}^+)$ function of $\delta( x )$ can be found in \cite{BEL}.

Now we recall some important elementary facts used in the study of fully non-linear geometric PDEs. Let $\lambda=(\lambda_1,\cdots,\lambda_n)\in \mathbb R^n$. Recall the $k$-th elementary symmetric functions of the vector $\lambda$ is defined as follows:
\begin{equation}
  \begin{array}[]{rll}
    \sigma_k(\lambda)=\sum_{1\le {i_1}<\cdots<{i_k}\le n}\lambda_{i_1}\cdots\lambda_{i_k}.
  \end{array}
  \label{sigma_k}
\end{equation}
In particular, $\sigma_1(\lambda)=\sum^n_{i=1}\lambda_i$ and $\sigma_n(\lambda) = \lambda_1\cdots\lambda_n$.

Below is a version of the well-known Newton-MacLaurin inequality for elementary symmetric functions which is the most important algebraic inequality in studying fully non-linear PDEs.
\begin{lemm}{\bf (Newton's Inequality \cite{N})}
  \label{lemm : key}
  Let $\lambda=(\lambda_1,\cdots,\lambda_n)$ with $\lambda_i>0$ for all $i=1,\cdots,n$ and $\sigma_k(\lambda)$ defined as in (\ref{sigma_k}). Then
  \begin{equation}
    \frac{\sigma_{n-1}(\lambda)}{\sigma_n(\lambda)}\ge\cdots\ge c(n,k)\frac{\sigma_{k-1}(\lambda)}{\sigma_k(\lambda)}\ge\cdots\ge n^2\frac{1}{\sigma_1(\lambda)},
    \label{lemm : key : equ}
  \end{equation}
  where $c(n,k)=\frac{n(n-k+1)}{k}$. The equalities hold if and only if $\lambda_1=\cdots=\lambda_n$.
\end{lemm}

Now we apply Lemma \ref{lemm : key} to prove the following proposition.
\begin{prop}
  \label{prop : mean convex}
  Let $\kappa=(\kappa_1, \cdots, \kappa_n)\in \mathbb R^n$ be the principal curvatures and $H$ the mean curvature
  of the boundary at a point on $\partial \Omega\in C^2$.  Then
  \begin{equation}
    \sum_i^n\frac{\kappa_i}{1-\delta\kappa_i}\ge \frac{nH}{n-\delta H},
    \label{prop : mean convex : equ}
  \end{equation}
whenever $1-\delta \kappa_i>0$ is satisfied for all $i=1, \cdots, n$. Equality holds if and only if $\kappa_1=\cdots=\kappa_n$.
\end{prop}
\begin{proof} Note that $\sigma_1(\kappa)=H$. Let $\lambda_i=1-\delta\kappa_i$, then $\sigma_1(\lambda)=n-\delta H$.
  We may assume that $\delta>0$, otherwise the result holds trivially. Applying (\ref{lemm : key : equ}), we have
  \begin{equation}
    \begin{array}{rll}
     \di \sum_i^n\frac{\delta\kappa_i}{1-\delta\kappa_i}=&\displaystyle\sum_{i=1}^n\frac{1-\lambda_i}{\lambda_i}=\displaystyle\sum_{i=1}^n\frac{1}{\lambda_i}-n\\
   \label{}
  \end{array}
  \end{equation}
  It is not hard to see that $\di\frac{\sigma_{n-1}(\lambda)}{\sigma_n(\lambda)}=\di\sum_{i=1}^n\frac{1}{\lambda_i}$. Hence, from (\ref{lemm : key : equ})
  \begin{equation}
    \begin{array}{rll}
     \di \sum_i^n\frac{\delta\kappa_i}{1-\delta\kappa_i} =&\di\frac{\sigma_{n-1}(\lambda)}{\sigma_n(\lambda)}-n\\
      \ge&\di\frac{n^2}{\sigma_1(\lambda)}-n\\
      =&\di\frac{n\delta H}{n-\delta H}.
    \label{}
  \end{array}
  \end{equation}
  By Lemma \ref{lemm : key}, equality holds if and only if all the $\lambda_i$s are the same. Equivalently, all the principal curvatures at the point must be equal.

\end{proof}

Combining Lemma \ref{hessian} and Proposition \ref{prop : mean convex}, and also applying Lemma \ref{positiveness}, one easily sees that (\ref{key theorem : equation}) holds on the good set $G$. 

\begin{coro}
 Let  $\Omega\subset \mathbb R^{n+1}$. Then for any $x\in G$,
\begin{equation}
  -\Delta \delta( x )\ge \frac{nH( y )}{n-\delta H( y )},
  \label{comparison good set}
\end{equation}
where $\delta( x ):=\inf_{ y \in \mathbb{R}^{n+1}\setminus \Omega} dist(x,y)$ and $H( y )$ is the mean curvature at the nearest point $y=N( x )\in \partial \Omega$ of $ x $.
\end{coro}

\section{Superharmonicity in the distribution sense}\label{distribution section}

\subsection{Proof of Theorem \ref{key theorem} when $\partial \Omega\in
C^{2,1}$}
{\it In this subsection,
we assume that $\partial \Omega$ is
$C^{2,1}$.}

Since
the test function $\varphi$ in
 (\ref{distribution sense})
has support in $B(0, R)$ for some $R>0$, we can
replace  $\Omega$  by
 a bounded
$\Omega_R$, still with $C^{2, 1}$ boundary, and 
$\Omega\cap B(0, 3R)=\Omega_R\cap B(0,3R)$.
It is clear that the distance function $\delta_R$, for $\Omega_R$,
coincides with the distance function $\delta$ on the support of $\varphi$.
Therefore we can assume that $\Omega$ is bounded in deriving
 (\ref{distribution sense})
for $\varphi$.

For $z\in \partial \Omega$, let
$$
\bar \rho(z): = \sup\{ t: z+t\eta(z)\in G\},
$$
where $\eta=-\nu$ is the inward unit normal. From every point $z$ on $\partial \Omega$, move along the inner normal until first hitting a point on the singular set $S$. We will denote this point to be $m(z)$ following the notations in \cite{LN}. It is known that
$$
m(z): =  z+\bar \rho(z)\eta(z).
$$
 The following non-trivial result was independently established, with different proofs, by Itoh-Tanaka \cite{IT} and Li-Nirenberg \cite{LN}. 

\begin{theo}
  \cite{IT,LN}  The map $m(z)$ and the function 
 $\bar\rho(z)$ are in  $C^{0,1}_{loc}(\partial \Omega)$.
  \label{LN theo}
\end{theo}

As a corollary of the above theorem, one obtains
\begin{coro}
  \cite{LN} Let $\Omega\subset \mathbb R^{n+1}$
and $S\subset\Omega$ be the singular set defined in (\ref{singular set}). The Hausdorff measure of the singular set $H^{n}(S)<\infty$. 
\end{coro}

 Recall the following fact. For $x\in G= \Omega\setminus S$, if we let $N(x)$ be the unique point on $\partial \Omega$, such that,
\[ \delta(x)= | x-N(x)|,\]
i.e., $N(x)$ is the nearest point
on $\partial \Omega$, then $\delta(x)\in C^2(\Omega\setminus S)$.

Next, we introduce the following normalized distance function $h(x)$ 
in $G$ which will be important later, 
\begin{equation}
  h(x):=\di\frac{\delta(x)}{\Lambda(x)}
  \label{normalized distance}
\end{equation}
where $\Lambda(x)=\bar\rho(N(x))$, $N(x)$ is the nearest point of $x$ on $\partial \Omega$ and $\bar\rho(z)$ is the Lipschitz function in Theorem \ref{LN theo}. Note $\Lambda(x)$, and therefore, $h(x)$, originally defined in $G$ can be extended as a continuous function in $\Omega=G\cup S$,
by defining the value of $h$ on $S$ to be $1$. Therefore $\Lambda$ and $h$, belong to $C^{0,1}_{loc}(\overline \Omega \setminus S)\cap C^0(\overline \Omega)$.

Indeed, we have
\begin{lemm}
  \label{h continuity}
For $\forall \bar x\in S$, 
\begin{equation}
  \di\lim_{x\rightarrow \bar x, x\in G}h(x)=1.
  \label{}
\end{equation}
\end{lemm}

\begin{proof}
  For $\bar x\in S$, $\exists \bar z\in \partial\Omega$, s.t., 
  \begin{equation}
    m(\bar z)=\bar z + \bar t \eta(\bar z)=\bar x,
    \label{}
  \end{equation}
where $\eta$ is the unit inner normal of $\partial \Omega$ at $\bar z$.
 We also have $|m(\bar z)-\bar z|=|\bar x-\bar z|=\bar t$.

 $\forall x_i\in G$, $x_i\rightarrow\bar x$, $\exists !
\ z_i:= N(x_i)\in \partial \Omega$, s.t. $ |x_i-z_i|=\delta(x_i)$.

By Corollary 4.11 of \cite{LN},
\begin{equation}
  \Lambda(x_i)>|x_i-z_i|=\delta(x_i),
  \label{}
\end{equation}
which implies 
\begin{equation}
  \di  \liminf_{i\rightarrow\infty}\Lambda(x_i)\ge\delta(\bar x).
  \label{}
\end{equation}

On the other hand, since
 $m(z_i)=z_i+t_i\eta(z_i)$, we have $\Lambda(x_i)=t_i$. We need the following claim :\\

  {\bf Claim:} 
  \begin{equation} 
  \limsup_{i\rightarrow \infty}\Lambda(x_i)\le \delta(\bar x).
  \label{}
\end{equation}
We now prove the claim by contradiction. If not, then $\exists \alpha>0$, s.t. $\Lambda(x_i)>\delta(\bar x)+\alpha$, for $\forall i$ large. Passing to a subsequence, we may assume that 
\begin{equation}
  \begin{array}[]{rll}
 z_i\rightarrow&\hat z\in\partial\Omega  \\
\Lambda(x_i)=t_i\rightarrow&\hat t\ge \delta(\bar x)+\alpha.
  \end{array}
  \label{}
\end{equation}
By the continuity of $m(z)$, c.f. \cite{LN}, 
\begin{equation}
  m(z_i)=z_i+t_i\eta(z_i)\rightarrow m(\hat z).
  \label{}
\end{equation}
We have 
\begin{equation}
  m(\hat z)=\hat z+\hat t\eta(\hat z).
  \label{6-1}
\end{equation}
But $x_i\rightarrow\bar x$, $x_i=z_i+\tilde t_i\eta(z_i)$, and $\tilde t_i<t_i$, $|x_i-z_i|=\tilde t_i$, $\tilde t_i\rightarrow \tilde t$. In the end we have
\begin{equation}
  \bar x = \hat z +\tilde t \eta(\hat z).
  \label{}
\end{equation}
Since
\begin{equation}
  \begin{array}[]{rll}
    \tilde t_i= & |x_i-z_i|\\
    =&dist(x_i,\partial \Omega)\\
    \le & dist(x_i,\bar x)+dist(\bar x,\partial\Omega)\\
    =&dist(x_i,\bar x)+\delta(\bar x)
  \end{array}
  \label{}
\end{equation}
where the term $dist(x_i,\bar x)\rightarrow 0$. This implies $\tilde t\le \delta(\bar x)\le \hat t -\alpha<\hat t$. By Corollary 4.11 of \cite{LN}, 
\begin{equation}
  \bar x=\hat z +\tilde x\eta(\hat z)\in G
  \label{}
\end{equation} 
in view of (\ref{6-1}), which yields a contradiction. Thus we have proved that 
\begin{equation}
  \di\lim_{i\rightarrow\infty}\Lambda(x_i)=\delta(\bar x),
  \label{}
\end{equation}
and
\begin{equation}
  \di\lim_{x\rightarrow \bar x}\Lambda(\bar x)=\delta(\bar x).
  \label{}
\end{equation}
The proof of the lemma is finished.

\end{proof}
$h(x)$ satisfies the following lemma.

\begin{lemm}
  The normalized distance function $h(x)\in C^{0,1}_{loc}(\overline \Omega \setminus S)\cap C^0(\overline \Omega)$, and
  \begin{equation}
     h(x)= \left\{
    \begin{array}[]{cll}
    &0, & x\in \partial \Omega,\\
    \\
    &1, & x\in S
    \end{array}
    \right.
    \label{}
  \end{equation}
  and $0<h(x)<1$ otherwise.
\end{lemm}

We consider
\begin{equation}
  h_{\epsilon}(x)=\di\int_{B(0,\epsilon)} h(x-y)\varphi_{\epsilon}(y)dy,
  \label{h epsilon}
\end{equation}
where $\varphi_{\epsilon}(x)=
\epsilon^{-n}\varphi(\frac x\epsilon)$ is a
 standard mollifier with compact support in a $\epsilon$-neighbourhood of $x$. From this definition, one has 
$$h_{\epsilon}\rightarrow h
\qquad\mbox{ in }\  C^0_{loc}(\bar\Omega\setminus S). 
$$

From now on, we fix $\forall 0<\mu<1$. Let us choose a sequence of $\lambda_{\e}$,
\begin{equation}
  \lambda_{\epsilon}\rightarrow 1-\mu, \mbox{ as } \e\rightarrow 0,
  \label{}
\end{equation}
such that $\lambda_{\epsilon}$ are  regular values of $h_{\epsilon}$.
It follows, for small $\epsilon$ (depending on $\mu$),  that
\begin{equation}
  \Sigma_{\epsilon}:=\{ h_{\epsilon} = \lambda_{\epsilon}\}\subset G
  \label{epsilon hypersurface : def}
\end{equation}
are regular smooth hypersurfaces.\\

We first show that the smooth hypersurfaces $\Sigma_{\epsilon}$
  stays away from 
the singular set $S$ and close to $\{x\in \Omega: h(x)= 1-\mu\}$
for small $\epsilon$.
\begin{lemm} Let
 $\Omega \subset\mathbb R^{n+1}$.
Then 
 the hypersurface $\Sigma_{\epsilon}$
 defined in (\ref{epsilon hypersurface : def}) satisfies 
 \begin{equation}
\lim_{\epsilon\to 0}
  dist\Big(\Sigma_{\epsilon}, \{x : h(x)=1-\mu\}\Big)=0.
  \label{epsilon hypersurface : equ2}
\end{equation}
\label{hypersurface family}
\end{lemm}
\begin{proof}  Suppose the contrary, there exists 
$\alpha>0$ such that for some 
$x_\epsilon\in \Sigma_\epsilon$,
\begin{equation}\label{A-1}
dist\Big(x_\epsilon,  \{ h=1-\mu\}\Big)\ge \alpha,
\end{equation}
along a sequence of $\epsilon\to 0$.
Passing to another subsequence, we may assume
that $x_\epsilon\to \bar x\in \mathbb R^{n+1}$.
By the continuity of $h$,
$$
h(\bar x)=\lim_{\epsilon\to 0}
h_\epsilon(x_\epsilon)=\lim_{\epsilon\to 0}
\lambda_\epsilon=1-\mu.
$$
It follows from (\ref{A-1}) that 
$|x_\epsilon-\bar x|\ge \alpha>0$,
violating the convergence of $x_\epsilon$ to $\bar x$.
\end{proof}

For every $0<\mu<1/8$, there exists, in view of Lemma
\ref{hypersurface family},  $0<\epsilon_1(\mu)$
such that
$$
\Sigma_\epsilon\subset \{ 1-\frac {5\mu}4 \le h\le
1
- \frac {3\mu}4 \}.
$$

The following is the key lemma.

\begin{lemm}\label{abc}
  For any fixed $0<\mu <1/8$ and $0<\epsilon\le
\epsilon_1(\mu)$, $\exists\ C(\mu)>0$
such that,
for 
$\epsilon>0$ small enough,
  \begin{equation}
    \eta_{\epsilon}\cdot \nabla \delta\ge C(\mu)>0
\quad\mbox{on}\  \Sigma_{\epsilon}.
    \label{positive}
  \end{equation}
\end{lemm}

\begin{proof}
  For any $x\in \Sigma_{\epsilon}$, we first give the following claim.\\ 

  {\bf Claim:} \ For
$0<\epsilon\le
\epsilon_1(\mu)$,
 $x\in \Sigma_\epsilon$, 
  \begin{equation} 
\nabla h_{\e}(x)\cdot \nabla\delta(x)
\ge C'(\mu)>0.
  \label{claim}
\end{equation}

Lemma \ref{abc} follows from 
 (\ref{claim}) as follows. 
Since
 $ \{ 1-\frac {3\mu}2 \le h\le
1- \frac \mu 2\}$  stays
positive distance away from the singular set $S$,
and $h$ is locally Lipschitz on $G\setminus
S$, we have
$|\nabla h|\le C''(\mu)$ on 
 $ \{ 1-\frac {3\mu}2 \le h\le
1- \frac \mu 2\}$.
Thus
\begin{equation}
  \begin{array}[]{rll}
    |h_{\e}(x)-h_{\e}(\tilde x)|\le & \Dint |h(x-y)-h(\tilde x -y)|\varphi_{\e}(y)\\
    \le & C''(\mu)\Dint |x-\tilde x|\varphi_{\e}(y)\\
    \le & C''(\mu)|x-\tilde x|,
  \end{array}
  \label{}
\end{equation}
and we have $ |\nabla h_{\e}(x)|\le C''(\mu)$.
 Then estimate (\ref{positive}) follows  from 
$\eta_{\epsilon}=\frac{\nabla h_{\epsilon}(x)}{|\nabla h_{\epsilon}(x)|}$ 
and (\ref{claim}).
\end{proof}

\vspace{.5cm}
\begin{proof}
  ({\bf Proof of Claim (\ref{claim}).}) 
From definition of $h_{\epsilon}$ in (\ref{h epsilon}), we have
  \begin{equation}
    h_{\epsilon}(x+t\nabla \delta(x))-h_{\epsilon}(x) = \di\int_{B(0,\epsilon)}\{ h(x-y+t\nabla \delta(x))-h(x-y)\}\varphi_{\epsilon}(y)dy.
    \label{proof : equ0}
  \end{equation}

Notice that, since $N(x)$ is $C^{1,1}$ off the singular set, 
\begin{equation}
  \begin{array}[]{rll}
    N(x-y+t\nabla\delta(x))=&N(x-y+t\nabla\delta(x-y)+O(t\epsilon))\\
    =&N(x-y+t\nabla\delta(x-y))+O(t\epsilon)\\
    =&N(x-y)+O(t\epsilon),
  \end{array}
  \label{}
\end{equation}
where $N(x-y+t\nabla\delta(x-y))=N(x-y)$, because $\nabla\delta (x-y)$ is the inward normal direction of $\partial \Omega$ at $N(x-y)$.

This yields
\begin{equation}
  \Lambda(x-y+t\nabla\delta(x))=
\bar\rho(N(x-y+t\nabla\delta(x)))=
\bar\rho(N(x-y))+O(t\epsilon)=
\Lambda(x-y)+O(t\epsilon),
  \label{proof : equ0.1}
\end{equation}
where we have used Theorem \ref{LN theo}
which asserts 
  that $\bar\rho$ is a Lipschitz map.

We also have
\begin{equation}
  \begin{array}[]{rll}
&\delta \big(x-y+t\nabla\delta(x)\big)\\
=   &	\big|\big(x-y+t\nabla\delta(x)\big)-N\big(x-y+t\nabla\delta(x)\big)
\big|\\
    =&\big|x-y+t\nabla\delta(x)-N\big(x-y+t\nabla\delta(x-y)\big)+O(t\epsilon)
\big|\\
    =&\big|x-y+t\nabla\delta(x)-N(x-y)\big|+O(t\epsilon).
  \end{array}
  \label{proof : equ0.2}
\end{equation}

For $x\in\Sigma_{\epsilon}$, and $|y|<\epsilon$, using (\ref{proof : equ0.1}) and (\ref{proof : equ0.2}), we have
\begin{equation}
  \begin{array}[]{rll}
    &h\big(x-y+t\nabla\delta(x)\big) - h(x-y)\\
    \\
    =&\di\frac{|(x-y)-N(x-y)+t\nabla\delta(x)|}{\Lambda(x-y)}+O(t\epsilon)-\frac{|(x-y)-N(x-y)|}{\Lambda(x-y)}\\
    \\
    =&\di\frac{|(x-y)-N(x-y)+t\nabla\delta(x)|-|(x-y)-N(x-y)|}{\Lambda(x-y)}+O(t\epsilon)\\
  \end{array}
  \label{proof : equ1}
\end{equation}

By definition, we have 
\begin{equation}
  (x-y)-N(x-y) = |(x-y)-N(x-y)|\cdot\nabla\delta(x-y).
  \label{proof : equ2}
\end{equation}

Applying (\ref{proof : equ2}) to (\ref{proof : equ1}), we have
\begin{equation}
  \begin{array}[]{rll}
    &h\big(x-y+t\nabla\delta(x)\big) - h(x-y)\\
    =&\di\frac{\Big|\big [|(x-y)-N(x-y)|+t\big]\cdot\nabla\delta(x)\Big|-|(x-y)-N(x-y)|}{\Lambda(x-y)}+O(t\epsilon)\\
    =&\di\frac{t}{\Lambda(x-y)}+O(t\epsilon)\\
    =&\di\frac{t}{\Lambda(x)+O(\epsilon)}+O(t\epsilon).
  \end{array}
  \label{proof : equ3}
\end{equation}

Combining (\ref{proof : equ3}) and (\ref{proof : equ0}), we have
   \begin{equation}
\nabla h_{\epsilon}(x)
\cdot \nabla \delta(x)=
\lim_{t\to 0}
\di\frac{h_{\epsilon}(x+t\nabla \delta(x))-h_{\epsilon}(x)}{t}
= \frac 1{ \Lambda(x) }+O(\epsilon).
\label{a1}
\end{equation}
Estimate (\ref{claim}) follow from the above.

\end{proof}

We now prove Theorem \ref{key theorem} for a
$C^{2,1}$ domain.

\begin{theo}
Let   $\Omega\subset \mathbb R^{n+1}$, $n\ge1$.
Then
\begin{equation}
  -\Delta \delta( x )\ge \frac{nH( N(x) )}{n-\delta H( N(x) )},
  \label{key theorem : equation 2}
\end{equation}
in the distribution sense, i.e., for any $\varphi\in C^{\infty}_0(\Omega)$, $\varphi\ge0$, we have
  \begin{equation}
    \Dint_{\Omega}\nabla\delta\nabla \varphi \ge 
\Dint_{\Omega}\frac{n(H\circ N)}{n-\delta (H\circ N)}\varphi.
    \label{}
  \end{equation}
\end{theo}
Note that the function $(H\circ N)(x) $ is well defined
for $x\in G$, so it is a well defined $L^\infty$ function since
$\Omega\setminus G$ is of zero Lebesgue measure.

\begin{proof}
By  Lemma \ref{hypersurface family}, 
we can construct,
using a standard diagonal sequence selection argument,
a sequence of $\lambda_\epsilon\to 1^-$ such that
$$
\Sigma_\epsilon:=
\{x\in \Omega\ :\ h_\epsilon(x)=\lambda_\epsilon\}\
\ \mbox{has}\ C^\infty\ \mbox{boundary},
$$
and 
$$
\Omega_\epsilon:= \{x\in \Omega\ :\ h_\epsilon(x)<\lambda_\epsilon\}
$$
satisfies
  $$
\di\bigcup_{\e>0}\Omega_{\e} = G,
$$
and
\begin{equation}
\eta_\epsilon \cdot \nabla \delta\ge 0,\qquad \mbox{on}\
\partial\Sigma_\epsilon,
\label{B-1}
\end{equation}
where $\eta_\epsilon$ is the unit outer normal
of the boundary of
$
\Omega_\epsilon$.

Since $\delta(x)$ is 
$C^2$  on $\bar\Omega_{\e}\subset G\cup \partial\Omega$,
 we may apply the Green's formula to obtain 
\begin{equation}
  \begin{array}[]{rll}
   \di \int_{\Omega_{\e}}\nabla\delta\nabla\varphi =& -\Dint_{\Omega_{\e}}\varphi\Delta\delta+\int_{\partial \Omega_{\e}}\varphi \frac{\partial \delta}{\partial \eta_{\e}} \\
   \ge & -\Dint_{\Omega_{\e}}\varphi\Delta\delta\\
   \ge &  \Dint_{\Omega_{\e}}\frac{nH\circ N}{n-\delta H\circ N}\varphi,
  \end{array}
  \label{laplacian : proof : equ1}
\end{equation}
where the last two inequalities follow from (\ref{B-1}) and (\ref{comparison good set}) respectively.

Letting $\e \rightarrow 0$ in (\ref{laplacian : proof : equ1}), we complete the proof.
 \end{proof}
 \vspace{0.5cm}

\subsection{ Proof of Theorem \ref{key theorem}} 

As in the previous subsection, we 
 can assume that $\Omega$ is bounded in deriving
 (\ref{distribution sense})
for $\varphi$.

For $z\in\partial \Omega$, let, as in 
\cite{LN},
$$
\tilde m(z)=z+\tilde \rho(z)\eta(z),
$$
where $\eta(z)$ denotes the
unit inner normal to
$\partial \Omega$ at $z$ and $\tilde \rho(z)>0$ is the largest number so that
$$
dist(z+t\eta(z), \partial \Omega)=t,\qquad \forall\ t\in (0, \tilde \rho(z)).
$$
By Lemma 4.2 of \cite{LN} ($C^2$ regularity of
$\partial \Omega$ is enough for the proof),
$\tilde \rho(z)\ge \bar  \rho(z)$.
This implies that 
\begin{equation}\label{aaa1}
B(m(z), \bar \rho(z))\subset \Omega,
\quad z\in \partial B(m(z), \bar\rho(z)),\qquad
\forall\ z\in \partial\Omega.
\end{equation}
\begin{lemm}\label{c1surface}
For every $h\in C^2(\partial \Omega)$ satisfying
$$
0<h(z)<\bar\rho(z),\quad z\in \partial \Omega,
$$
let
$$
\Sigma:=\{z+h(z)\eta(z)\ |\
z\in\partial \Omega\}.
$$
Then $\Sigma$ 
is a $C^1$ hypersurface with
\begin{equation}
\eta^\Sigma(x)\cdot \nabla \delta(x)> 0,\qquad \forall\ x\in \Sigma,
\label{bbb}
\end{equation}
where $\eta^\Sigma(x)$ denotes the unit outer normal of
the boundary of $$
\{z+th(z)\eta(z)\ |\ z\in \partial\Omega, 0<t<1\}.
$$
\end{lemm}
\begin{proof}
For a point $z\in \partial \Omega$, we may assume without loss
of generality that $\bar\rho(z)=1)$. After a translation and rotation,
we may assume that $z=0$ is the origin, and the boundary near
$0$ is given by
$$
x_{n+1}=g(x'),\ \ x'=(x_1, \cdots, x_n),
$$
where $g$ is a $C^2$ function near $0'$ satisfying
$$
g(0')=0,\ \nabla g(0')=0,\ \
(\nabla^2g(0'))\ \mbox{is a diagonal matrix}.
$$
The unit inner normal to $\partial \Omega$
at $(x', g(x')$  near
$0$ is given by the graph of 
$$
\eta(x'):= 
\frac{  (-\nabla g(x'), 1) }
{  \sqrt{  1+|\nabla g(x')|^2  } }.
$$
The set $\Sigma$ is given locally by
$$
X(x'):= (x', g(x'))+
\tilde h(x')\eta(x'),
$$
where
$\tilde h(x')=h(x', g(x'))$ is a $C^2$ function near $0'$.
We know that $\tilde h(0')<\bar\rho(z)=1$.
Clearly $X\in C^1$.  We need to show that
$\Sigma$ indeed has a tangent plane at $X(0')$.

Using notations $e_1=(1, 0, \cdots, 0), ...,
e_{n+1}=(0, \cdots, 0, 1)$, we have, for $1\le \alpha\le n$,
$$
\frac {\partial X}{\partial x_ \alpha}(0')=
e_\alpha+ 
\tilde h_{x_\alpha}(0')e_{n+1}
+\tilde h(0')
\frac {\partial \eta}{ \partial x_\alpha}(0')
=[1-\tilde h(0') g_{x_\alpha x_\alpha}(0')]e_\alpha
+ 
\tilde h_{x_\alpha}(0')e_{n+1}.
$$
By (\ref{aaa1}) and $\bar\rho(z)=1$,
the unit ball centered at $e_{n+1}$ lies in
$\{ x_{n+1}\ge g(x')\}$ near $0$.  It follows that
$g_{x_\alpha x_\alpha}(0')\le 1$. 
Thus
 \begin{equation} \label{ggg}
1-\tilde h(0')g_{ x_\alpha x_\alpha}(0')>0.
 \end{equation}
It follows that $\Sigma$ has a tangent plane at
$X(0')$.
Since $  \tilde  \rho(z) \ge  \rho(z)=1$, we have
$$
 \delta(te_n)=t, \qquad \forall \
0<t<1,
$$
and therefore
$$
 \nabla  \delta(te_n)=e_n,
 \ 0<t<1.
$$
Since $ \eta^ \Sigma(h(0)e_n)$ is the outer normal to
the set, and $ \gamma(t):=th(0)e_n$
belongs to the set for $0<t<1$, we have
$$
  \eta^ \Sigma(h(0)e_n)  \cdot  \nabla  \delta(h(0)e_n)=
\eta^ \Sigma(h(0)e_n)  \cdot e_n
=  \frac 1{  h(0)  }
 \eta^ \Sigma(h(0)e_n)  \cdot      \gamma'(1)  \ge 0.
$$
Moreover, in view of (\ref{ggg}), 
$$
span \ \{  \frac {\partial X}{\partial x_ \alpha}(0') \}
=span \  \{ e_ \alpha+a_ \alpha e_n \},
 \quad  \mbox{for some constants} \ a_ \alpha,
$$
which does not contain $e_n$.
 The inequality  (\ref{bbb}) follows.
\end{proof}

For $\epsilon>0$ small, we construct $\bar\rho_\epsilon
\in C^2(\partial \Omega)$ satisfying
$$
|\bar\rho_\epsilon(z)-\bar\rho(z)|\le \epsilon \bar\rho(z),
\qquad\forall\ z\in \partial \Omega.
$$
Then we let
$$
\Sigma_\epsilon:=\{z+(1-\epsilon)\bar\rho(z)\eta(z)
\ |\ z\in \partial \Omega\}
$$
and
$$
\Omega_\epsilon:=\{z+t (1-\epsilon)\bar\rho_\epsilon(z)\eta(z)\
|\ z\in \partial \Omega, 0<t<1\}.
$$
Clearly,
$\partial \Omega_\epsilon=\Sigma_\epsilon\cup \partial \Omega$.
By Lemma \ref{c1surface},
$\Sigma_\epsilon$ is a $C^1$ hypersurface satisfying
$$
\eta_\epsilon\cdot \nabla \delta\ge 0,\qquad
\mbox{on}\  \Sigma_\epsilon,
$$
where $\eta_\epsilon$ is the unit outer normal of
$\partial \Omega_\epsilon$.
Clearly 
$$
\cup_{\epsilon>0}\Omega_\epsilon=G.
$$
With the above, the proof of
Theorem \ref{key theorem} for $C^{2,1}$
domain $\Omega$ in the previous subsection goes through without
any change, using the fact that $S$ has zero Lebesgue measure.

\subsection{}

Next, we prove the following proposition.

\begin{prop}
  Let $\Omega\subset\mathbb R^{n+1}$, $n\ge1$, and let
$G$ be the good set defined as in (\ref{singular set}). Then
  \begin{equation}
    \di\inf_{ x  \in  G}(-\Delta \delta(x)) 
= \di\inf_{ y \in \partial\Omega}H( y ),
    \label{}
  \end{equation}
  where $H( y )$ is the mean curvature of the boundary at $ y $.
  \label{prop : inf equi}
\end{prop}

 \begin{proof}
   From Lemma \ref{hessian}, we have
   \begin{equation}
     -\Delta\delta( x )=\sum^n_{i=1}\frac{\kappa_i(N(x))}{1-\delta(x)
\kappa_i(N(x))},\qquad x\in G.
     \label{aaa}
   \end{equation}
Since 
$\sum^n_{i=1}\frac{\kappa_i}{1-\delta\kappa_i}$
 is a nondecreasing
 function of $\delta$ independent of the sign of $\kappa_i$,
as long as $1-\delta\kappa_i>0$ for all $i$,
we have, in view of (\ref{aaa}),
$$
-\Delta \delta(x)
\ge 
\sum^n_{i=1} \kappa_i(N(x))=H(N(x))\ge
\di\inf_{ y \in \partial\Omega}H( y ), \qquad x\in G.
$$
It follows that
  $$
\di\inf_{ x  \in  G}(-\Delta \delta(x))
\ge  \di\inf_{ y \in \partial\Omega}H( y ).
$$
On the other hand,
for every $y\in \partial \Omega$,
since $x_t=y+t\nu(y)\in G$ for $t>0$ small,
we have, in view of (\ref{aaa}), 
$$
  \di\inf_{ x  \in  G}(-\Delta \delta(x))
\le \lim_{t\to 0^+}(-\Delta\delta(x_t))
= H(y).
$$
Thus 
$$
  \di\inf_{ x  \in  G}(-\Delta \delta(x))
\le\di\inf_{ y \in \partial\Omega}H( y ).
$$
Proposition \ref{prop : inf equi} is proved.
\end{proof}

As a direct corollary, we prove the equivalence theorem, Theorem \ref{equivalence theorem}.
\begin{proof}
  {\bf (Proof of Theorem \ref{equivalence theorem})} By definition, if $\delta( x )$ is superharmonic, then $-\Delta\delta( x )\ge0$, for any $ x \in G$. If $\Omega$ is weakly mean convex, then $H( y )\ge0$, for any $ y \in \partial \Omega$. Then the proof follows directly from Proposition \ref{prop : inf equi}.
\end{proof}

\begin{rema}
  Geometrically, $-\Delta \delta( x )$ is the mean curvature of the level surface of $\delta$ through $ x $
  at $ x $, see, e.g., Gilbarg-Trudinger \cite{GT}. The geometric interpretation of
  Theorem \ref{equivalence theorem} is that, the level surface of $\delta$ is mean convex through
  $ x \in\Omega$ if and only if the boundary is mean convex. The comparison between level surface and boundary is
  evident since $-\Delta \delta$ is a monotonically increasing function as $\delta\to 0$ along the perpendicular
  direction, which is true even when near points have negative principal curvature.
\end{rema}

\begin{rema}
  Another estimate on $-\Delta\delta$ can be found in Proposition \ref{estimate on Delta} which states that the growth of $-\Delta\delta$ with respect to $\delta$ is at least a polynomial growth of degree $p-1$.
\end{rema}

  \section{Proofs of Main theorems}\label{main results}

  We first observe the following identity.

\begin{lemm}
  \label{hardy identity}
  Let $\Omega\subset\mathbb R^{n+1}$.
 For any $f\in C^{\infty}_0(\Omega)$, the following holds
  \begin{equation}
    \begin{array}{rll}
    \Dint_{\Omega}|\nabla f|^2 d x  - \di \frac{1}{4}\Dint_{\Omega}\frac{f^2}{\delta^2}d x
    =\Dint_{\Omega}\Big|\nabla f-\frac{f\nabla \delta}{2\delta}\Big|^2d x+  \Dint_{\Omega}\nabla\delta\nabla \frac{f^2}{2\delta}dx.
    \label{hardy identity : equ}
  \end{array}
  \end{equation}
\end{lemm}

\begin{proof}

Since $f\in C^{\infty}_0(\Omega)$, $\di\frac{f^2}{\delta}$ is a Lipschitz function compactly supported in $\Omega$. We have
\begin{equation}\label{}
  \begin{array}{rll}
     \Dint_{\Omega}\nabla\delta\nabla \frac{f^2}{2\delta}dx
    =&-\Dint_{\Omega}\frac{f^2|\nabla \delta|^2}{2\delta^2}d x +  \di\int_{\Omega}\frac{{f}\nabla \delta\cdot\nabla f}{\delta}d x\\
=& -\Dint_{\Omega}\frac{f^2}{2\delta^2}d x +  \di\int_{\Omega}\frac{{f}\nabla \delta\cdot\nabla f}{\delta}d x ,
   \end{array}
   \label{divergence}
\end{equation}
where the last step follows from $|\nabla \delta|=1$ a.e. in $\Omega$. Using the elementary identity
$$
|X|^2-|Y|^2=|X-Y|^2+2<X,Y>-2|Y|^2,
$$
and letting $X=\nabla f$, $Y=\frac{f}{2\delta}\nabla \delta$, we have the following pointwise identity,
\begin{equation}
  \begin{array}[]{rll}
    |\nabla f|^2 - \frac{f^2|\nabla\delta|^2}{4\delta^2} =& |\nabla f-\frac{f\nabla \delta}{2\delta}|^2 + \frac{{f}\nabla f\cdot \nabla\delta}{\delta}-\frac{f^2|\nabla \delta|^2}{2\delta^2}.\\
  \end{array}
  \label{}
\end{equation}
Upon integration we have
$$
  \begin{array}[]{rll}
    & \di\int_{\Omega}|\nabla f|^2d x  - \int_{\Omega}\frac{f^2}{4\delta^2}d x \\
    =& \di\int_{\Omega}|\nabla f-\frac{f\nabla \delta}{2\delta}|^2d x+\di\int_{\Omega}\frac{{f}\nabla \delta\cdot\nabla f}{\delta}d x-\Dint_{\Omega}\frac{f^2}{2\delta^2}d x\\
    =&\di\int_{\Omega}|\nabla f-\frac{f\nabla \delta}{2\delta}|^2d x+\Dint_{\Omega}\nabla\delta\nabla \frac{f^2}{2\delta}dx,\\
  \end{array}
  \label{}
$$
where we have used (\ref{divergence}) in the last step.
\end{proof}

We now prove Theorem \ref{sharp hardy inequality} ($p=2$) and Theorem \ref{main theorem}.

\begin{proof}{\bf (Proof of Theorem \ref{sharp hardy inequality} ($p=2$) and Theorem \ref{main theorem})}
  By
 Theorem \ref{key theorem},
and a standard density argument,
   \begin{equation}
    \begin{array}[]{rll}
 \Dint_{\Omega}\nabla\delta\nabla \frac{f^2}{2\delta}dx \ge  \Dint_{\Omega}\frac{nH}{n-\delta H}\frac{f^2}{2\delta}dx.
    \end{array}
    \label{distribution form}
  \end{equation}
Applying (\ref{distribution form}) to (\ref{hardy identity : equ}) in Lemma \ref{hardy identity}, we have
  \begin{equation}
    \begin{array}{rll}
    \Dint_{\Omega}|\nabla f|^2 d x  - \di \frac{1}{4}\Dint_{\Omega}\frac{f^2}{\delta^2}d x    \ge &\Dint_{\Omega} \frac{nH}{n-\delta H}\frac{f^2}{2\delta}d x+\Dint_{\Omega}\Big|\nabla f-\frac{f\nabla \delta}{2\delta}\Big|^2d x\\
    \ge&\Dint_{\Omega} \frac{nH}{n-\delta H}\frac{f^2}{2\delta}d x.
    \label{main theorem : proof : equ1}
  \end{array}
  \end{equation}
  Now, suppose $H_0>0$. Otherwise, when $H_0=0$, the theorems hold trivially from (\ref{main theorem : proof : equ1}). Let $\phi(t):=\frac{1}{at-t^2}$, with $a>0$. First, we have the following elementary inequality,
 \begin{equation}
   \phi(t)\ge \frac{4}{a^2}, \quad \mbox{for all}\ t\in (0,a)
   \label{}
 \end{equation}
 since the minimum of $\phi(t)$ is attained at $t_0=\frac{a}2$.

Let $a=\frac{n}{H}$ and $t=\delta$. For
$\forall x\in\Omega\setminus S$, the fact that $t<a$ in this
case follows from (\ref{positivity}). Consequently, we have $ \di\frac{H}{(n-\delta H)\delta}\ge \frac{4H^2}{n^2} $ for $x\in G$ and
\begin{equation}
   \begin{array}{rll}
\Dint_{\Omega} \frac{nH}{n-\delta H}\frac{f^2}{2\delta}d x\ge \Dint_{\Omega}  \frac{2}{n}H^2f^2d x
   \label{main theorem : proof : equ2}
 \end{array}
 \end{equation}
  where we have used $\Omega\setminus G$ has measure zero.

 Apply (\ref{main theorem : proof : equ2}) to (\ref{main theorem : proof : equ1}), we have
  \begin{equation}
    \begin{array}{rll}
    \Dint_{\Omega}|\nabla f|^2 d x  - \di \frac{1}{4}\Dint_{\Omega}\frac{f^2}{\delta^2}d x
    \ge&\Dint_{\Omega} \frac{2}{n}H^2f^2d x\\
   \ge&\di\frac{2}{n}H^2_0\Dint_{\Omega}f^2 d x.
 \end{array}
    \label{}
  \end{equation}
 This finishes the proof of improved Hardy inequality in Theorem \ref{main theorem} with $\lambda(n,\Omega)\ge\frac{2}{n}H^2_0$, which also implies the Hardy inequality (\ref{sharp hardy inequality : equ1}) for $p=2$.
\end{proof}

We next prove the $L^p$ version of Hardy inequalities. First we give an inequality as a $L^p$ version of Lemma \ref{hardy identity}. In the context of convex domains, the following method was used first in \cite{FMT}.
\begin{lemm}
  \label{hardy : lemma 2}
  Let $\Omega\subset\mathbb R^{n+1}$.
 For any $f\in C^{\infty}_0(\Omega)$ and $p>1$, the following holds
  \begin{equation}
    \begin{array}{rll}
\Dint_{\Omega}|\nabla f|^p d x -\di\Big( \frac{p-1}p \Big)^p\Dint_{\Omega}\frac{|f|^p}{\delta^p}d x
\ge   \di\Big( \frac{p-1}p \Big)^{p-1}\Dint_{\Omega}\nabla\delta\nabla\frac{|f|^p}{\delta^{p-1}}d x.
    \label{hardy : lemma 2 : equ}
  \end{array}
  \end{equation}
\end{lemm}

\begin{proof}
Recall the following elementary inequality for vectors when $p>1$,
\begin{equation}
|X|^p-|Y|^p\ge p|Y|^{p-2}<X-Y,Y>.
\label{elementary inequality}
\end{equation}
Let $X=\nabla f$, $Y=\frac{p-1}{p}\frac{f}{\delta}\nabla \delta$, then the following pointwise identity holds
\begin{equation}
  \begin{array}[]{rll}
    |\nabla f|^p -\big(\frac{p-1}{p}\big)^p \frac{|f|^p|\nabla\delta|^p}{\delta^p} \ge \big(\frac{p-1}{p}\big)^{p-1}|\nabla\delta|^{p-2}\nabla \delta\nabla \frac{|f|^p}{\delta^{p-1}},
  \end{array}
  \label{hardy : lemma 2 : proof : equ1}
\end{equation}
where we have used that $X-Y=\delta^\frac{p-1}{p}\nabla \frac{f}{\delta^\frac{p-1}{p}}$.
Using the fact that $|\nabla\delta|=1$, we finished the proof upon integration.
\end{proof}

Next we derive a Brezis-Marcus type of improved $L^p$ Hardy inequality on weakly mean convex domains.
\begin{theo}\label{improved hardy : Lp}
  Let $\Omega\subset\mathbb R^{n+1}$, $n\ge1$.
 Suppose $\Omega$ is {\em weakly mean convex}, then for any $f\in C^{\infty}_0(\Omega)$, $p>1$, the following holds:
  \begin{equation}
    \Dint_{\Omega}|\nabla f|^p d x \ge\di\Big( \frac{p-1}p \Big)^p\Dint_\Omega\frac{|f|^p}{\delta^p}d x  + \lambda(n,p,\Omega) \int_{\Omega}|f|^pd x .
    \label{3.20}
  \end{equation}
  where $\lambda(n,p,\Omega)=\di\Big(\frac{p-1}p\Big)^{p-1}\inf_{\Omega\setminus S}\frac{-\Delta \delta}{\delta^{p-1}}\ge\frac{p}{n^{p-1}}H_0^p$.
  \label{Lp}
 \end{theo}

\begin{proof}
  The proof is similar as in the proof of the $L^2$ version. By the same reasoning as in (\ref{distribution form}) and (\ref{hardy : lemma 2 : equ}), one simply observes that
  \begin{equation}
    \begin{array}[]{rll}
      \Dint_{\Omega}\nabla\delta\nabla \frac{f^p}{\delta^{p-1}}dx \ge  \Dint_{\Omega}\frac{nH}{n-\delta H}\frac{f^p}{\delta^{p-1}}dx.
    \end{array}
    \label{distribution form 2}
  \end{equation}

Applying (\ref{distribution form 2}) to (\ref{hardy : lemma 2 : equ}) in Lemma \ref{hardy : lemma 2}, we have
\begin{equation}
  \begin{array}{rll}
    \Dint_{\Omega}|\nabla f|^p d x -\di\Big( \frac{p-1}p \Big)^p\Dint_{\Omega}\frac{|f|^p}{\delta^p}d x \ge \di\Big( \frac{p-1}p \Big)^{p-1}\Dint_{\Omega}\frac{nH}{n-\delta H}\frac{|f|^p}{\delta^{p-1}}d x.
  \label{Lp : proof : equ1}
\end{array}
\end{equation}

  Let $\phi(t):=\frac{1}{at^{p-1}-t^p}$, with $p>1$ and $a>0$. First, we have the following elementary inequality,
 \begin{equation}
   \phi(t)\ge \frac{p}{a^p}(\frac{p}{p-1})^{p-1},\quad t\in (0,a),
   \label{}
 \end{equation}
 since the minimum of $\phi(t)$ for $t\in (0,a)$ is attained at $t_0=a\frac{p-1}p$.

 Suppose $H_0>0$, otherwise $H_0=0$ and the proof is finished. Let $a=\frac{n}{H}$ and $t=\delta$, then we have $ \di\frac{nH}{(n-\delta H)\delta^{p-1}}\ge \frac{pH^{p}}{n^{p-1}}(\frac{p}{p-1})^{p-1}$ for $x\in G$ and
\begin{equation}
   \begin{array}{rll}
     \Dint_{G} \frac{nH}{n-\delta H}\frac{|f|^p}{\delta^{p-1}}d x\ge& \Big (\frac{p}{p-1}\Big)^{p-1}\frac{p}{n^{p-1}}\Dint_{G} H^{p}|f|^pd x\\
\ge& \Big (\frac{p}{p-1}\Big)^{p-1}\frac{p}{n^{p-1}}H_0^p\Dint_{G} |f|^pd x.\\
   \label{Lp : proof : equ2}
 \end{array}
 \end{equation}

Using the fact that $\Omega\setminus G$ has measure zero and applying (\ref{Lp : proof : equ2}) to (\ref{Lp : proof : equ1}), we finish the proof.
 \end{proof}

The next result may be of independent interest.
\begin{prop}
  \label{estimate on Delta}
  Let $\Omega\subset \mathbb R^{n+1}$.
 Suppose $\partial \Omega$ is weakly mean convex. Let $H_0:=\inf_{\partial\Omega}H( y )$ and $\delta( x )$ be the distance function to the boundary, then for $p>1$, and $\forall x\in \Omega\setminus S$,
  \begin{equation}
    -\Delta\delta( x )\ge\frac{pH^p( y )}{n^{p-1}}\Big(\frac{p}{p-1}\Big)^{p-1}\delta^{p-1}( x )\ge \frac{pH^p_0}{n^{p-1}}\Big(\frac{p}{p-1}\Big)^{p-1}\delta^{p-1}( x ),
    \label{}
  \end{equation}
  where $ y =N( x )\in \partial\Omega$ is the near point of $ x $.
\end{prop}
\begin{proof}
  When $\delta=0$ the proof follows from Theorem \ref{equivalence theorem}.
  Suppose $\delta>0$. When $p=2$, the proof of
  \begin{equation}
    \frac{-\Delta\delta( x )}{\delta( x )}\ge\frac{4}{n}H^2( y )
    \label{}
  \end{equation}
can be found in the proof of Theorem \ref{sharp hardy inequality} ($p=2$) and \ref{main theorem}. For general $p>1$, the proof of
  \begin{equation}
    \frac{-\Delta\delta( x )}{\delta^{p-1}( x )}\ge\frac{pH^p_0}{n^{p-1}}\Big(\frac{p}{p-1}\Big)^{p-1}
    \label{3.27}
  \end{equation}
  can be found in the proof of Theorem \ref{improved hardy : Lp} after using inequality
  (\ref{key theorem : equation}).
\end{proof}

As mentioned in the introduction, the geometric requirement of weakly mean convexity cannot be weakened for sharp Hardy-type inequalities. However, by adding an extra positive term to the left hand-side of the inequality, one can still prove a Hardy type inequality for general domains. In particular, we have the following inequality for domains with boundaries that have points of negative mean curvature.

\begin{theo}
  Let $\Omega\subset\mathbb R^{n+1}$, $n\ge1$. Suppose $H_0:=\inf_{ y \in\partial \Omega}H( y )<0$. Then for any $f\in C^{\infty}_0(\Omega)$, with $p>1$, the following holds
    \begin{equation}
  \di\int_{\Omega}|\nabla f( x )|^pd x + \Big(\frac{p-1}p\Big)^{p-1}|H_0|\Dint_{\Omega}\frac{|f|^p}{\delta^{p-1}}d x  \ge c(n,p,\Omega)\di\int_{\Omega}\frac{|f|^p}{\delta^p}d x ,
    \label{arbitrary domains : equ1}
  \end{equation}
  where $c(n,p,\Omega)= \Big(\frac{p-1}p\Big)^{p}$ is the same constant as in the mean convex case.
  \label{arbitrary domains}
\end{theo}
\begin{proof}
  The proof is the same as in the mean convex case. One only need to notice that the function $\frac{nH_0}{n-\delta H_0}$ is monotonic with respect to $\delta$ for any fixed $H_0\in \mathbb R$, then the following holds on the good set $G$
  \begin{equation}
    \frac{nH_0}{n-\delta H_0}\ge H_0.
    \label{mono}
  \end{equation}
  Applying (\ref{mono}) to (\ref{Lp : proof : equ1}), proceed as before, we obtain
  \begin{equation}
  \Dint_{\Omega}|\nabla f|^p d x \ge\di\Big( \frac{p-1}p \Big)^p\Dint_\Omega\frac{|f|^p}{\delta^p}d x +\Big(\frac{p-1}p\Big)^{p-1}H_0\Dint_{\Omega}\frac{|f|^p}{\delta^{p-1}}d x .
  \label{arbitrary domains : equ2}
\end{equation}
When $H_0<0$, we complete the proof by moving the last term to the left hand-side.
\end{proof}

In the rest of this section, we discuss the sharpness of the boundary geometric condition of weakly mean convexity. We will show by examples that the $H\ge0$ condition cannot be weakened.\\

{\bf Example.} (Exterior domain) Let $\Omega_{\epsilon}:=\mathbb R^{n+1}\backslash B_{\frac{n}{\epsilon}}$ where
$B_{\frac{n}{\epsilon}}$ is a ball with radius $\frac{n}{\epsilon}$ centered at the origin.
The mean curvature of the boundary with respect to the exterior domain $\Omega_{\epsilon}$ is $H\equiv-\epsilon$.
For $\mu_p(\Omega)$ given in (\ref{best constant : hardy}), we use an idea of Marcus, Mizel, and Pinchover to show that
the best constant $\mu_{n+1}(\Omega_{\epsilon})=0$ for each $\epsilon>0$,
see Example 2 in \cite{MMP}.

Consider the sequence of domains $\Omega_k=\frac{1}{k}\Omega_{\epsilon}$, $k\ge1$.
Then, as shown in \cite{MMP}, $\mu_{n+1}(\Omega_k)=\mu_{n+1}(\Omega_{\epsilon})$. On the other hand, by
Lemma~12 of \cite{MMP}, $\lim\sup_{k\rightarrow\infty}\mu_{n+1}(\Omega_k)\le \mu_{n+1}(R^{n+1}_{*})$, where
$\mathbb R^{n+1}_*=\mathbb R^{n+1}\backslash \{0\}$.
According to Example~1 in \cite{MMP}, $\mu_p(\mathbb R^n_*)=|\frac{n-p}{p}|^p$. Thus $\mu_{n+1}(\Omega_{\epsilon}) =0$.
In particular, this example shows that for each $\epsilon>0$  the Hardy inequality (\ref{sharp hardy inequality : equ1}) does not hold on $\Omega_{\epsilon}\subset \mathbb R^2$ having negative mean curvature $-\epsilon$ on the boundary.

In \cite{AL}, Avkhadiev and Laptev construct ellipsoid shells, i.e. two ellipsoids $E_1$, $E_2$, $\overline E_2\subset E_1\subset R^{n+1}$ with $n\ge2$, and show that the sharp Hardy inequality fails on $\Omega:=E_1\backslash \overline E_2$. We can rescale $\Omega$ in such a way that the mean curvature $H( y )\ge-\epsilon$ for all $ y  \in \partial E_2$ with arbitrary $\epsilon>0$ and $H>0$ on $\partial E_1$. Then we have another example which indicates that Hardy inequality with the sharp constant $c(n,p,\Omega)=(\frac{p-1}p)^p$ does not hold in general when the boundary has negative mean curvature.

\section{Other important inequalities on mean convex domains}\label{applications}
Due to the fundamental role that $-\Delta\delta$ plays in Hardy
type inequalities, we can apply the inequality
in Theorem \ref{key theorem} to prove other inequalities. For example, in
\cite{FMT2}, Filippas, Maz'ya, and Tertikas proved several critical Hardy-Sobolev inequalities. As a special case of
their Theorem~5.3, the following holds:\\

 \label{FMT critical hardy sobolev}{\bf Theorem. (Filippas-Maz'ya-Tertikas \cite{FMT2}}) {\em Let $2\le p<n$, $p<q\le \frac{np}{n-p}$, and $\Omega\subset \mathbb R^{n}$ be a bounded domain with $C^2$ boundary. If the distance function $\delta( x ) $ is superharmonic, i.e. $-\Delta \delta\ge 0$, then there exists a positive constant $c=c(\Omega)$ such that for all $u\in C^{\infty}_0(\Omega)$, there holds
  \begin{equation}
    \di\int_{\Omega}|\nabla u|^pd x -(\frac{p-1}p)^p\int_{\Omega}\frac{|u|^p}{\delta^p}d x \ge c\Big(\di\int_{\Omega}\delta^{-q+\frac{q-p}{p}n}|u|^qd x \Big)^{\frac{p}{q}}.
    \label{FMT critical hardy sobolev : equ1}
  \end{equation}
}

As a direct corollary of our Theorem \ref{key theorem}, we can generalize the above theorem to weakly mean convex domains.

\begin{theo}
  \label{FMT critical hardy sobolev : mean convex} Let $2\le p<n$, $p<q\le \frac{np}{n-p}$, and $\Omega\subset \mathbb R^{n}$.
 If the domain is weakly mean convex, then there exists a positive constant $c=c(\Omega)$ such that for all $u\in C^{\infty}_0(\Omega)$, there holds
  \begin{equation}
    \di\int_{\Omega}|\nabla u|^pd x -(\frac{p-1}p)^p\int_{\Omega}\frac{|u|^p}{\delta^p}d x \ge c\Big(\di\int_{\Omega}\delta^{-q+\frac{q-p}{p}n}|u|^qd x \Big)^{\frac{p}{q}}.
   \label{FMT critical hardy sobolev : mean convex : equ1}
  \end{equation}
\end{theo}
In the extreme case, where $q=\frac{np}{n-p}$,
the right-hand side is precisely the critical Sobolev term.

\begin{rema}
  \label{hardy sobolev : rema} Notice that for $k=1$ the Condition (C) in \cite{FMT2} is equivalent to weakly mean convexity because of Theorem \ref{equivalence theorem} and a bounded $C^2$ domain satisfies Condition (R) in \cite{FMT2}.
\end{rema}

Sobolev inequalities with a sharp Hardy term as in (\ref{FMT critical hardy sobolev : equ1}) have drawn much attention
recently. But the best constant $c$ for the Sobolev term is largely unknown for general domains. If the domains are an
upper half plane or a ball, the best constants are estimated by Tertikas and Tintarev in \cite{TT} for $n>3$ and by Benguria,
Frank, and Loss in \cite{BFL} for $n=3$. When $n=2$, a Hardy-Moser-Trudinger inequality is given by Wang and Ye \cite{WY}.
Recently, Frank and Loss \cite{FL} proved that the constant $c(\Omega)$ in
(\ref{FMT critical hardy sobolev : equ1}) with $q=\frac{np}{n-p}$ can be replaced by a constant $c$ which
is independent of $\Omega$ provided the domain $\Omega$ is convex.

Applying Theorem \ref{equivalence theorem} to Theorem 3.4 in \cite{FMT2}, we can extend the Hardy-Sobolev inequality to weakly mean convex domains.

\begin{theo}\label{hardy sobolev}{\bf (Hardy-Sobolev-Maz'ya Inequality)}
   Let $\Omega\subset\mathbb R^{n+1}$, $n\ge1$ be a domain with a $C^{2}$ boundary. If $\partial \Omega$ is weakly mean convex, then there exists a positive constant
  $C=C(n,p,\Omega)$ such that for any $u\in C^{\infty}_0(\Omega)$
  \begin{equation}
    \di\int_{\Omega}|\nabla u|^pd x  -\Big(\frac{p-1}{p}\Big)^p\int_{\Omega}\frac{|u|^p}{\delta^p}d x \ge C\Big(\di\int_{\Omega}|u|^\frac{np}{n-p}d x \Big)^\frac{n-p}n.
    \label{}
  \end{equation}
\end{theo}

\section{Examples}\label{examples}
In this section, we will sample several interesting examples with non-trivial topology, which, of course, are
non-convex.\\

{\bf Example 1.}\label{example 0} Let $\Omega\subset \mathbb R^{n+1}$ be the critical ring torus with minor radius $r=1$ and major radius $R=2$. $H\ge0$ on the boundary $\partial \Omega$. From elementary differential geometry textbooks, e.g., \cite{O}, we can easily calculate all the principal curvatures as below.

\begin{equation}
  \kappa_1=1, \kappa_2=\frac{\cos(\theta)}{2+\cos(\theta)}.
  \label{}
\end{equation}

For simplicity, we denote $a:=\cos(\theta)$. We observe that

\begin{equation}
  \mu(\delta,a):=\frac{-\Delta\delta}{2\delta}= \Big(\frac{1}{1-\delta} +\frac{\frac{a}{2+a}}{1-\frac{a}{2+a}\delta}\Big)\frac{1}{2\delta},
  \label{}
\end{equation}
is a monotonically increasing function of $a$. Hence, for any fixed $\delta$,
\begin{equation}
  \begin{array}[]{rll}
\mu(\delta,a)\ge&\mu(\delta,-1)\\
=&  \Big(\frac{1}{1-\delta} - \frac{1}{1+\delta}\Big)\frac{1}{2\delta}\\
\ge&1
  \end{array}
  \label{}
\end{equation}
which yields $\lambda(n,\Omega) \ge 1$.  Since on the inner equator, $\frac{-\Delta\delta}{2\delta}=1$, we have $\lambda(n,\Omega)=1$.\\

{\bf Example 2.}\label{example 1} Let $\Omega\subset \mathbb R^{n+1}$ be the ring torus with minor radius $r$ and major radius $R$, see the same example when $n=2$ in \cite{BEL}. If $R-2r>0$, then $H>0$ everywhere on the boundary $\partial \Omega$.\\

{\bf Example 3.} \label{example 2} Other examples include mean convex torus with higher genus. \\

{\bf Example 4.}\label{example 3} It is geometrically clear that one can perturb the examples in Example 2 and Example 3 slightly and still keep the mean curvature strictly positive on the boundary $\partial \Omega$. \\

{\bf Example 5.}\label{example 4}
Another example is a domain enclosed by a parabola in a plane or enclosed by an paraboloid in $\mathbb R^{n+1}$, $n\ge2$.
This domain is convex with infinite interior radius. \\

{\bf Example 6.} Lastly, there are domains with an embedded minimal surface as boundary. We notice that Hardy-type inequalities hold on either side of the minimal surface. We now prove Corollary \ref{minimal surface : coro}.

\begin{proof}
  {\bf (Proof of Corollary \ref{minimal surface : coro}.)} For a point $\mathbf y\in \partial\Omega$, let $\kappa_1$, $\kappa_2$ be the principal curvatures. Since $H=\kappa_1+\kappa_2=0$, we can denote the two principal curvatures to be $\kappa$ and $-\kappa$ with $\kappa\ge0$. Since the following inequality holds everywhere in the good set $G$ and on the whole domain $\Omega$ in the distribution sense, we have
  \[
  \frac{-\Delta\delta}{2\delta}=\frac{\kappa^2}{1-\kappa^2\delta^2}\ge \kappa^2.
  \]
  Let $\kappa_0:=\di\inf_{\mathbf y\in\mathcal M}|\kappa(\mathbf y)|$, applying Theorem \ref{main theorem}, and the proof is complete.
\end{proof}

\section{Acknowledgment}  The authors would like to thank G\"unter Stolz and Rudi Weikard for organizing the UAB
Math-Physics seminar where the current work was motivated. Thanks to Pengfei Guan and Michael Loss for their interests. The first and second authors would like to thank Rupert Frank for helpful discussions and for pointing out the oversights in an earlier draft of this paper.

\bigskip
\bigskip

\noindent
{\it Note added}\ 
We are pleased to acknowledge receipt of 
the preprint
``$L^1$ Hardy inequalities
with weights''  by Georgios Psaradakis after completing the work in this paper.
The interesting results of Psaradakis are closely related, but with the emphasis upon results in $L^1$.



\begin{thebibliography}{99}
  \bibitem{AK} Armitage, D.H.; Kuran, \"U. {\em The convexity of a domain and the superharmonicity of the signed distance function.} Proc. Amer. Math. Soc. 93 (1985), no. 4, 598-600.

  \bibitem{AL} Avkhadiev, F.; Laptev, A. {\em Hardy inequalities for nonconvex domains. } Around the research of Vladimir Maz'ya. I, 1-12,
Int. Math. Ser. (N. Y.), 11, Springer, New York, 2010.

  \bibitem{AW} Avkhadiev, F.; Wirths, K. {\em Unified Poincar\'e and Hardy inequalities with sharp constants for convex domains.}
ZAMM Z. Angew. Math. Mech. 87 (2007), no. 8-9, 632-642.

  \bibitem{BEL} Balinsky, A.A.; Evans, W.D.; Lewis, R.T.. {\em Hardy's inequality and curvature}, Math. Phys. Arc. 11-2
  and arXiv.org, 2011.

\bibitem{BFT} Barbatis, G.; Filippas, S.; Tertikas, A. {\em A unified approach to improved $L^p$ Hardy inequalities with best constants.} Trans. Amer. Math. Soc. 356 (2004), no. 6, 2169-2196


  \bibitem{BFL} Benguria, R.D.; Frank, R.L.; Loss M. {\em The sharp constant in the Hardy-Sobolev-Maz’ya inequality in the
    three dimensional upper half-space}, Math. Res. Lett. 15 (2008), 613–622.


  \bibitem{BM}  Brezis, H.; Marcus, M. {\em Hardy's inequalities revisited.} Dedicated to Ennio De Giorgi. Ann. Scuola Norm. Sup. Pisa Cl. Sci. (4) 25 (1997), no. 1-2, 217-237 (1998).

  \bibitem{D} Davies, E. B. {\em The Hardy constant.} Quart. J. Math. Oxford Ser. (2) 46 (1995), no. 184, 417-431.

  \bibitem{D1} Davies, E.B. {\em Spectral theory and differential operators.}  Cambridge Studies in Advanced Mathematics, 42. Cambridge University Press, Cambridge, 1995. x+182 pp. ISBN: 0-521-47250-4

  \bibitem{D2} Davies, E.B. {\em A review of Hardy inequalities}. The Maz'ya anniversary collection, Vol. 2 (Rostock, 1998), 55-67,
Oper. Theory Adv. Appl., 110, Birkh\"auser, Basel, 1999.

  \bibitem{DD} Davila, J.; Dupaigne, L. {\em Hardy-type inequalities.} J. Eur. Math. Soc. (JEMS) 6 (2004), no. 3, 335-365.

\bibitem{DE} P. Duclos; P. Exner. {\em Curvature-induced bound states in quantum waveguides in two and three dimensions.}
Rev. Math. Phys. 7 (1995), 73-102.

  \bibitem{EE} Edmunds, D.E.; Evans, W.D. {\em Hardy operators, Function spaces, and Embeddings},  Springer Monographs in Mathematics. Springer-Verlag, Berlin, 2004. xii+326 pp. ISBN: 3-540-21972-2
  \bibitem{EL} Evans, W.D.; Lewis, R.T. {\em Hardy and Rellich inequalities with remainders}. J. Math. Inequal. 1 (2007), no. 4, 473-490.


  \bibitem{FMT1} Filippas, S.; Maz'ya, V.; Tertikas, A. {\em A sharp Hardy Sobolev inequality.} C. R. Math. Acad. Sci. Paris 339 (2004), no. 7, 483-486.

   \bibitem{FMT} Filippas, S.; Maz'ya, V.; Tertikas, A. {\em On a question of Brezis and Marcus.} Calc. Var. Partial Differential Equations 25 (2006), no. 4, 491-501.

  \bibitem{FMT2} Filippas, S.; Maz'ya, V.; Tertikas, A. {\em Critical Hardy-Sobolev inequalities.} J. Math. Pures Appl. (9) 87 (2007), no. 1, 37-56.

\bibitem{FMoT} Filippas, S.; Moschini, L.; Tertikas, A. {\em Sharp two-sided heat kernel estimates for critical
  Schr\"odinger operators on bounded domains.} Comm. Math. Phys. 273 (2007), no. 1, 237-281.

  \bibitem{FT} Filippas, S.; Tertikas, A. {\em Optimizing improved Hardy inequalities.} J. Funct. Anal. 192 (2002), no. 1, 186-233.


  \bibitem{FL} Frank,Rupert L.; Loss, Michael. {\em Hardy-Sobolev-Maz'ya inequalities for arbitrary domains. } arXiv:1102.4394.

  \bibitem{GT} Gilbarg, David; Trudinger, Neil S. {\em Elliptic partial differential equations of second order.}
Reprint of the 1998 edition. Classics in Mathematics. Springer-Verlag, Berlin, 2001. xiv+517 pp.


\bibitem{H} E.~M. Harrell II. {\em Commutators, eigenvalue gaps, and mean curvature in the theory of Schr\"odinger
operators.} Communications in Partial Differential Equations 32(2007), 401--413.

\bibitem{HL} E.~M. Harrell II.; Loss, Michael. {\em On the Laplace operator penalized by mean curvature.} Communications in Mathematical
Physics 195(1998), 643--650.

\bibitem{GR} N. Ghoussoub; F. Robert. {\em The effect of curvature on the best constant in the Hardy-Sobolev inequallities.}
Geom. Funct. Anal. 16(2006), 1201--1245.

\bibitem{HHL} Hoffmann-Ostenhof, M.; Hoffmann-Ostenhof, Th.; Laptev, A. {\em A geometrical version of Hardy's inequalities.} J. Funct. Anal. 189 (2002), no. 2, 539-548.

\bibitem{IT}  Itoh, J.-I.; Tanaka, M. {\em The Lipschitz continuity of the distance function to the cut locus}, Trans. Amer. Math. Soc. 353 (2001),
21-40.

\bibitem{LN} Li, Yanyan; Nirenberg, Louis. {\em The distance function to the boundary, Finsler geometry, and the singular set of viscosity solutions of some Hamilton-Jacobi equations.} Comm. Pure Appl. Math. 58 (2005), no. 1, 85–146.

\bibitem{LN1} Li, Yanyan; Nirenberg, Louis. {\em Regularity of the distance function to the boundary}. Rend. Accad. Naz. Sci. XL Mem. Mat. Appl. (5) 29 (2005), 257–264.

\bibitem{MMP} Marcus, M.; Mizel, V.J.; Pinchover, Y. {\em On the best constant for Hardy's inequality in $\mathbb R^n$}. Trans. Amer. Math. Soc. 350 (1998), no. 8, 3237-3255.

  \bibitem{MS} Matskewich, T.; Sobolevskii. {\em The best possible constant in a generalized Hardy's inequality for convex domains in $\mathbb R^n$.} Nonlinear Anal. 28 (1997), no. 9, 1601-1610.

  \bibitem{N} Newton, Isaac. {Arithmetica universalis: sive de compositione et resolutione arithmetica liber.},
  (1707).
  \bibitem{O} O'Neill, B. {\em Elementary differential geometry}. Revised second edition. Elsevier/Academic Press, Amsterdam, 2006. xii+503

  \bibitem{R} Royden, H. L. {\em Real analysis.} Third edition. Macmillan Publishing Company, New York, 1988. xx+444 pp. ISBN: 0-02-404151-3 

  \bibitem{TT} Tertikas, A.; Tintarev, K. {\em On existence of minimizers for the Hardy-Sobolev-Maz’ya inequality}, Ann. Mat.
Pura Appl. 186 (2007), 645-662.

 \bibitem{Th}Thorpe, J. {\em Elementary topics in differential geometry.} Undergraduate Texts in Mathematics. Springer-Verlag, New York, 1994. xiv+253 pp.

  \bibitem{T} Tidblom, J. {\em A geometrical version of Hardy's inequality for $W^{1,p}_0(\Omega)$}. Proc. Amer. Math. Soc. 132 (2004), no. 8, 2265-2271



  \bibitem{VZ} Vazquez, J.L.; Zuazua, E. {\em The Hardy inequality and the asymptotic behaviour of the heat equation with an inverse-square potential.}. J. Funct. Anal. 173 (2000), no. 1, 103-153.

  \bibitem{WY} Wang, Guofang; Ye, Dong. {\em A Hardy-Moser-Trudinger inequality}, arXiv:1012.5591.
\end{thebibliography}
\end{document}